\documentclass{eptcs}
\usepackage{amsmath,amsfonts,amssymb}
\usepackage{times}
\usepackage{graphicx}
\usepackage{amsmath}
\usepackage{url}
\usepackage{amsfonts}
\usepackage{amssymb}
\newtheorem{theorem}{Theorem}

\newtheorem{definition}[theorem]{Definition}

\newtheorem{remark}[theorem]{Remark}

\newenvironment{proof}[1][Proof]{\textbf{#1.} }{\qed}
\def\qed{{~~}\vskip-2.5ex\hfill$\Box$\vskip2pt}
\def\authorrunning{Fauser, Raynaud, Vickers}
\def\titlerunning{Born rule in spectral bundles}
\let\frak=\mathfrak

\begin{document}

\title{The Born rule as structure of spectral bundles\\(extended abstract)}
\author{Bertfried Fauser, Guillaume Raynaud, Steven Vickers\\
School of Computer Science, University of Birmingham,\\
Birmingham, B15 2TT, UK\\{}
[B.Fauser$\vert$G.Raynaud$\vert$S.J.Vickers]@cs.bham.ac.uk}
\maketitle
\begin{abstract}
Topos approaches to quantum foundations are described in a unified way by
means of spectral bundles, where the base space is a space of contexts and
each fibre is its spectrum. Differences in variance are due to the bundle
being a fibration or opfibration. Relative to this structure, the
probabilistic predictions of the Born rule in finite dimensional settings are
then described as a section of a bundle of valuations. The construction uses
in an essential way the geometric nature of the valuation locale monad.
\end{abstract}

%%%%%%%%%%%%%%%%%%%%%%%%%%%%%%%%%%%%%%%%%%%%%%%%%%%%%%%%%%%%
%
%

\section{Introduction}

Two topos approaches to quantum foundations, \cite{DoeringIsham:WhatThingTTFP}
and \cite{HeunenLandsmanSpitters:ToposAQT}\cite{HeunenLandsmanSpitters:Bohrn},
describe a quantum system given algebraically (as von Neumann algebra or
C*-algebra respectively) as a topos combined with a space defined internally
``in'' the topos. However, toposes present many difficulties to the beginner:
their basic definitions are non-trivial, and many important parts of
topos-theory relate only indirectly to the basic definitions. The aim of this
paper is to describe the topos approaches -- both existing and prospective --
in terms of ``spectral bundles'' $p:\Sigma\rightarrow B$.\footnote{We shall
use the word ``bundle'' in a very general sense, of arbitrary map between two
topological spaces (more precisely, the spaces need to be point-free, but the
naive reader can largely ignore this issue). If ``bundle'' just means the same
as ``map'' one might wonder why we should waste a second word on the same
notion. However, they will carry different connotations. A bundle
$p:X\rightarrow B$ is to be thought of as a family of spaces (the fibres
$p^{-1}(\{b\})$) parameterized by base point $b$. This is exactly the view one
has of, for example, the tangent bundle over a differentiable manifold.} We
shall refer to the base points $C$ in $B$ as \emph{contexts}, or (to use a
phrase of \cite{DoeringIsham:WhatThingTTFP}) ``classical points of view'', and
each fibre $C^{\ast}\Sigma=p^{-1}\{C\}$ as the \emph{spectrum} of $C$, also
written $\Sigma_{C}=\operatorname*{Spec}(C)$. It is a ``classical state
space'' from point of view $C$. As we shall see below, a canonical realization
of this in a quantum situation is where $C$ represents a commuting set of
observables and then the state space is the set of their common eigenspaces.
Each of the observables can be diagonalized with respect to those eigenspaces,
and then the diagonal entries, the corresponding eigenvalues, are the measured
values, while the resulting state is got by projecting to the corresponding eigenspace.

In the topos approaches, the topos is the topos $\mathcal{S}B$ of sheaves over
$B$ and the bundle corresponds to something internal in that topos: either an
object or a point-free topological space (locale). Since our aim is to replace
the language of toposes by that of spaces, one might wonder if there is still
any need at all for topos theory in this topos approach. The key insight is
that the study of \emph{bundles over} $B$ is equivalent to the study of
\emph{spaces internal in} $\mathcal{S}B$.\footnote{``Space'' here implicitly
means point-free.} In other words, the study of bundles is just a version of
topology, but different from ordinary classical topology since it has to be
adapted to the non-classical internal mathematics of toposes. As we shall see,
under certain logical constraints of \emph{geometricity} this becomes
equivalent to treating bundles as ``fibrewise topology'', ordinary topology
but sprinkled with base point parameters $C$ from $B$.

In the present paper this is the key to our description of the Born rule. We
use a construction that, for any space $X$, gives a \emph{valuation space}
$\frak{V}X$ whose points are the valuations (regular measures) on $X$. Since
this is geometric, it can be applied fibrewise to any bundle $p:X\rightarrow
B$ to give a valuation bundle $q:\frak{V}_{B}X\rightarrow B$. A valuation on a
fibre of the spectral bundle, in other words a valuation on the spectrum of
some context, turns out to be exactly the kind of probability distribution
that is expressed in the Born rule. It is topos theory -- more precisely, the
understanding of the logic of topos-internal mathematics and of the
geometricity constraints -- that gives us access to this fibrewise valuation
space and allows us to infer that it has good properties.

For other treatments of the Born rule in the topos approach
see~\cite{Doering:QuStMSP,DoeringIsham:ClassicalQPTV} .

%%%%%%%%%%%%%%%%%%%%%%%%%%%%%%%%%%%%%%%%%%%%%%%%%%%%%%%%%%%%
%
%
%
%

\subsection{States}

\paragraph{Classical physics:}

Let us be clear about the notions of \emph{state} that will concern us. In
classical physics, it is assumed that (given a selection of observables $O$)
there is a set $\Sigma$ of (classically) \emph{pure states} that determine the
values of all the observables. Thus each observable $O$ is realized
mathematically as a function $\tilde{O}:\Sigma\rightarrow\mathbb{R}$. To
measure $O$ in state $x$ is then to discover the value $\tilde{O}(x)$. In
practice, as for instance in thermodynamics, we often do not have access to
the exact pure states and are reduced to using probabilistic distributions. We
shall write $\frak{V}^{(1)}(X)$ for the space of probability valuations
(regular measures with total mass 1) on $X$, assuming that $X$ has appropriate
structure (specifically: $X$ will be a point-free topology). In that case $O$
is also realized as $\frak{V}^{(1)}(\tilde{O}):\frak{V}^{(1)}(\Sigma
)\rightarrow\frak{V}^{(1)}(\mathbb{R})$, taking \emph{mixed states} (points of
$\frak{V}^{(1)}(\Sigma)$) to distributions of reals (points of $\frak{V}%
^{(1)}(\mathbb{R})$). In other words, once $m\in\frak{V}^{(1)}(\Sigma)$ is
given, $\tilde{O}$ becomes a \emph{random variable}. Although in practice the
mixed states of this probabilistic approach may be the best we can do, it is
nonetheless assumed that they do arise as probability distributions of pure states.

\paragraph{Quantum physics:}

In quantum physics, on the other hand, an observable $O$ may be realized as a
self-adjoint operator $\hat{O}$ on some Hilbert space $\mathcal{H}$, which we
assume, for simplicity, is of finite dimension $n$. The spectral theorem then
tells us that $\hat{O}=\sum_{i=1}^{m}\lambda_{i}P_{i}$ where the $\lambda_{i}%
$s are the eigenvalues of $\hat{O}$ and the $P_{i}$s form a complete (they sum
to $1$) set of mutually orthogonal projectors (self-adjoint idempotents),
projecting onto the corresponding eigenspaces. The \emph{quantum} pure states
are now taken to be the non-zero vectors $|{\phi}\rangle\in\mathcal{H}$,
modulo scalar multiplication (in other words the states are 1-dimensional
subspaces, or \emph{rays}, in $\mathcal{H}$). To measure $O$ in state $|{\psi
}\rangle$ is probabilistic. Its measured result is one of the eigenvalues
$\lambda_{i}$, with probability $\frac{\langle{\psi}\mid{P_{i}|\psi}\rangle
}{\langle{\psi}\mid{\psi}\rangle}$ according to the Born rule, and with
resulting state $P_{i}|{\psi}\rangle$ according to the L\"{u}ders Principle.
Quantum mixed states (probabilistic distributions over quantum pure states)
are still of use, but now even the pure states have a probabilistic nature.

This raises the question of whether there might be \emph{classically} pure
states out of which the quantum pure states are mixed. The answer depends on
the observables. If we consider $O$ alone, then we can take $\Sigma
=\{1,\ldots,m\}$ (for the $m$ eigenvalues, taken as distinct). Then $\tilde
{O}(i)=\lambda_{i}$ and $|{\psi}\rangle$ is the mixed state in which $i$ has
weight $\frac{\langle{\psi}\mid{P_{i}|\psi}\rangle}{\langle{\psi}\mid{\psi
}\rangle}$. This extends to the situation where we have any collection of
commuting observables, since they are simultaneously diagonalizable and there
is a complete set of orthogonal projectors $P_{i}$ ($1\leq i\leq m$) such that
each $\hat{O}$ can be expressed as $\sum_{i=1}^{m}\lambda_{i}P_{i}$ (although
now we cannot necessarily assume that the $\lambda_{i}$s are distinct).

However, the Kochen-Specker Theorem tells us that in general when we have
\emph{non-}commuting observables there is no possible space of classically
pure states out of which the quantum pure states can be mixed.

We shall in general take ``context'' to mean some situation, such as a family
of commuting observables, in which it is possible to find a classically pure
state space -- the spectrum of the context. A fundamental aspect (going back
to~\cite{IshamBfield:ToposPKST1} and~\cite{IshamBfield:ToposPKST2}) of the
topos approach to quantum physics is to work in a mathematics that works in
contexts but in some sense works in them all at once. In this mathematics,
though it is \emph{logically} non-classical, there is some possibility of
being \emph{physically} classical.

%%%%%%%%%%%%%%%%%%%%%%%%%%%%%%%%%%%%%%%%%%%%%%%%%%%%%%%%%%%%
%
%
%
%

\subsection{Toposes and bundles}

The big content of topos theory is that it provides a generalization of
topological spaces, new spaces whose topological structure must be given by
stipulating the \emph{sheaves,} not just the opens. However, the topos
approach to quantum physics as currently conceived uses only ungeneralized
spaces (though point-free, as locales), and one of our aims here is to conduct
our discussion in terms of those spaces.

What toposes bring is a more conscious use of sheaves, and in particular the
ability to manipulate them by interpreting mathematics (subject to
constructivist constraints) as the internal mathematics of $\mathcal{S}B$.

To get a feel for how this works, think of an open $U$ of a space $B$ as a
``continuously varying truth value'', parametrized by points $b$ of $B$. The
value is true when $b\in U$: so $U$ is a generalized truth value that says not
\emph{whether} something is true, but \emph{where}. There is an asymmetry,
``locality of truth'', deriving from the nature of openness: if the value is
true at $x$ then it is true throughout some neighbourhood of $x$. The same
does not hold for falsehood, and this asymmetry shows up in the associated
logic -- negation cannot be a connective. Another way to see the situation,
which respects this asymmetry, is as a ``continuously varying subsingleton''.
If we write $\ast$ for the unique element of some standard singleton set, then
the subsingleton is $\{\ast\}$ where $b\in U$, $\emptyset$ where $b\notin U$.
Thus locality of truth translates into locality of existence of $\ast$. This
view is simply that of the inclusion $U\hookrightarrow B$ as bundle, since the
fibre at $b$ is exactly the subsingleton described.

Sheaves can be understood as generalizing ``continuously varying
subsingletons'' to ``continuously varying sets''. In bundle form, when one
incorporates not only locality of existence of elements but also locality of
equality between them, one gets the notion of \emph{local homeomorphism}
$X\rightarrow B$, the fibres being the continuously varying sets. (Note that
the definition implies that the fibres, as subspaces of $X$, all have the
discrete topology.)

The topos of sheaves over $B$ is (equivalent to) the topos of local
homeomorphisms with codomain $B$. Many mathematical constructions can also be
carried through on sheaves, and this interpretation gives an \emph{internal
mathematics} of the topos. For reasons such as the locality of truth, it does
not obey all reasoning principles of classical logic and set theory, but
nonetheless it is intuitionistic. The methodology of manipulating sheaves by
reasoning intuitionistically has proved an effective one. Moreover, there is a
particularly important \emph{geometric} fragment comprising those
constructions that work fibrewise on the local homeomorphisms. For these the
intuitionistic features are less obtrusive.

The final step is to move to general bundles as ``continuously varying
spaces''. Just as local homeomorphisms are (by definition) ``sets'' in the
internal mathematics, we should also like general bundles to be spaces there.
This idea works well, with two provisos. First, the spaces need to be
point-free, and, second, the logic needs to be geometric in order for it to
work fibrewise, so that the whole construction varies continuously with the
base points. Getting the mathematics to work within the geometricity
constraints is non-trivial, and the present work provides a case study in a
more general geometrization programme. However, it has the beneficial effect
of enabling point-based reasoning for point-free topology.

The key result is what we shall call the \emph{Localic Bundle Theorem}
(\cite{FourmanScott:SheavesLogic}, \cite{JoyalTier}), which states that frames
in the internal mathematics of a topos are dual equivalent to localic
geometric morphisms with that topos as codomain. Restricted to the case where
the topos is that of sheaves over a locale $B$, these are equivalent to locale
maps with codomain $B$, in other words localic bundles over $B$.

This has an important consequences for constructions on locales. If the
construction is topos-valid then it can be applied to internal frames in
toposes of sheaves and hence gives a construction on localic bundles. We shall
be particularly interested in constructions that are \emph{geometric} in the
sense that, when applied to bundles, are preserved by pullback, since this
implies that they work fibrewise.

\emph{To summarize:} In this work we shall be working with bundles, and with
constructions on spaces that can be applied fibrewise to the bundles. Although
the details will be largely hidden, the justification for the fibrewise
construction, and for choosing the appropriate topology on the bundle space,
will be that the constructions work point-free in a way that is topos-valid
and moreover geometric. This perspective on geometric logic is summarized in
\cite{Vickers:ContIsGeom}.

\subsection{Bundles as fibrations and opfibrations}

Given a bundle $p:\Sigma\rightarrow B$ one can ask how the fibres interact
with specialization order amongst the base points: if $C\sqsubseteq D$, is
there a corresponding map, in either direction, between the fibres $C^{\ast
}\Sigma$ and $D^{\ast}\Sigma$? In general there is no such map, but there are
special classes of bundles, known as fibrations and opfibrations, for which
there are. The general theory \cite{Street} works in an arbitrary 2-category
and has been examined in \cite{Jo:FibPP} in the 2-category of toposes. Our own
interest is in the restriction to the 2-category of locales, where the
category enrichment is in fact order enrichment, the specialization order on
each homset: if $f,g:X\rightarrow Y$ then $f\sqsubseteq g$ if $f^{\ast}V\leq
g^{\ast}V$ for all $V$ open in $Y$.

For a fibration $p$, for $C\sqsubseteq D$ in $B$ there is a map
contravariantly between the fibres, $D^{\ast}\Sigma\rightarrow C^{\ast}\Sigma
$, while for an opfibration it is covariant. Moreover, in both cases the fibre
maps are characterized universally in a way that determines them uniquely. We
shall describe this in a way that deals with the points $C\sqsubseteq D$
generically, allowing for generalized points. As we shall see later, these
notions provide a fundamental explanation for the difference in variance seen
in two topos approaches to quantum theory.

Such pairs $C\sqsubseteq D$ are classified by the exponential locale
$B^{\mathbb{S}}$ where $\mathbb{S}$ is the Sierpinski locale with points
$\bot\sqsubseteq\top$.\footnote{Potentially it has other points too, but they
arise only in non-classical mathematics.} Given $f:\mathbb{S}\rightarrow B$, a
point of $B^{\mathbb{S}}$, we have $f(\bot)\sqsubseteq f(\top)$ in $B$, and
this gives two maps $\pi_{\bot}\sqsubseteq\pi_{\top}:B^{\mathbb{S}}\rightarrow
B$. They are generic. For any other $C\sqsubseteq D:W\rightarrow B$, there is
a unique $f:W\rightarrow B^{\mathbb{S}}$ such that $C=\pi_{\bot}\circ f$ and
$D=\pi_{\top}\circ f$. This allows us to understand the points of
$B^{\mathbb{S}}$ as the pairs $C\sqsubseteq D$. Because of this we don't ask
about fibre maps for arbitrary pairs $C\sqsubseteq D$, but just for the
generic pair $\pi_{\bot}\sqsubseteq\pi_{\top}$. A fibre map found there can
then be pulled back for arbitrary $C\sqsubseteq D$.

Over $B^{\mathbb{S}}$ we have two bundles $\pi_{\bot}^{\ast}\Sigma$ and
$\pi_{\top}^{\ast}\Sigma$, and so far there is no reason why there should be a
map between them. However, we do have a span over them from the bundle
$p^{\mathbb{S}}:\Sigma^{\mathbb{S}}\rightarrow B^{\mathbb{S}}$. For example,
the commutative square%
\[%
\begin{array}
[c]{rll}%
\Sigma^{\mathbb{S}} & \overset{\pi_{\bot}}{\longrightarrow} & \Sigma\\
p^{\mathbb{S}}\downarrow &  & \downarrow p\\
B^{\mathbb{S}} & \underset{\pi_{\bot}}{\longrightarrow} & B
\end{array}
\]
gives us a map $\rho_{\bot}:\Sigma^{\mathbb{S}}\rightarrow\pi_{\bot}^{\ast
}\Sigma$ over $B^{\mathbb{S}}$. Similarly, for $\top$ we get a map
$\lambda_{\top}:\Sigma^{\mathbb{S}}\rightarrow\pi_{\top}^{\ast}\Sigma$ over
$B^{\mathbb{S}}$.

\begin{definition}
With the notation as above, $p$ is a \emph{fibration} if $\lambda_{\top}$ has
a right adjoint $\rho_{\top}$ over $B^{\mathbb{S}}$, with its counit an equality,
and is an
\emph{opfibration} if $\rho_{\bot}$ has a left adjoint $\lambda_{\bot}$ over
$B^{\mathbb{S}}$, with its unit an equality.
\end{definition}

To say that $\rho_{\top}$ is right adjoint of $\lambda_{\top}$ is to say that
$\operatorname*{Id}_{\Sigma^{\mathbb{S}}}\sqsubseteq\rho_{\top}\circ
\lambda_{\top}$ (the unit) and $\lambda_{\top}\circ\rho_{\top}\sqsubseteq
\operatorname*{Id}_{\pi_{\top}^{\ast}\Sigma}$ (the counit), and similarly for $\rho_{\bot}$
and $\lambda_{\bot}$. In these two cases we get fibre maps. For a fibration we
get, contravariantly, $\rho_{\bot}\circ\rho_{\top}:\pi_{\top}^{\ast}%
\Sigma\rightarrow\pi_{\bot}^{\ast}\Sigma$, while for an opfibration we get,
covariantly, $\lambda_{\top}\circ\lambda_{\bot}:\pi_{\bot}^{\ast}%
\Sigma\rightarrow\pi_{\top}^{\ast}\Sigma$.%
\[%
\begin{array}
[c]{lllll}%
\pi_{\bot}^{\ast}\Sigma &
\begin{array}
[c]{l}%
\overset{\lambda_{\bot}}{\dashrightarrow}\\
\underset{\rho_{\bot}}{\longleftarrow}%
\end{array}
& \Sigma^{\mathbb{S}} &
\begin{array}
[c]{l}%
\overset{\lambda_{\top}}{\longrightarrow}\\
\underset{\rho_{\top}}{\dashleftarrow}%
\end{array}
& \pi_{\top}^{\ast}\Sigma
\end{array}
\]

\begin{theorem}
\label{FibOpfibThm}Let $p:\Sigma\rightarrow B$ be a bundle as above.

\begin{enumerate}
\item  If $p$ is a local homeomorphism, thus corresponding to an object of
$\mathcal{S}B$, then it is an opfibration.

\item  If $p$ corresponds to a compact regular locale in $\mathcal{S}B$, then
it is a fibration.
\end{enumerate}
\end{theorem}

\begin{proof}
\cite{Jo:FibPP}
\end{proof}

\subsection{Valuation locales}

Standard measure theory works badly in toposes, suffering from deep
set-theoretic problems. For many purposes a satisfactory replacement can be
found by replacing measurable spaces and measures by locales $X$ and
\emph{valuations} $m$ on them. Such an $m$ is a Scott continuous map from the
frame of opens $\Omega X$ to the lower reals $\overrightarrow{[0,\infty]}$,
satisfying the \emph{modular law} $m(U\vee V)+m(U\wedge V)=m(U)+m(V)$ and also
$m(\emptyset)=0$. (The \emph{lower} reals differ constructively from the
Dedekind reals in being approximable from below but not from above. For
present purposes it is best to understand them as being given the Scott
topology instead of the usual Hausdorff topology.) \emph{Probability
valuations} are those for which $m(X)=1$.

For every locale $X$ there can be constructed a \emph{valuation locale}
$\frak{V}X$ whose points are the valuations on $X$; it has a sublocale
$\frak{V}^{(1)}X$ whose points are the probability valuations. These were
first defined in~\cite{Heckmann:ProbPDISL}, following ideas of the
probabilistic powerdomain of \cite{JonesPlot:ProbPower}, and were further
developed in \cite{Integration} and \cite{CoqSpit:IntVal}. In particular the
results of \cite{CoqSpit:IntVal} were central in the quantum treatment of
\cite{HeunenLandsmanSpitters:ToposAQT}.

More recent work \cite{Vickers:Riesz} has shown that $\frak{V}$ and
$\frak{V}^{(1)}$ are the functor parts of monads, localic analogues of the
Giry monad in measure theory \cite{Giry:CatProb} and the distribution monad of
\cite{AbrBrand:UnifiedSTANLC}. The monads are commutative, meaning that
product valuations exist and a Fubini Theorem holds.

\cite{Vickers:Riesz} also describes in some detail the geometricity of
$\frak{V}$ (and likewise $\frak{V}^{(1)}$), and this will be key to our
development here. The topos-validity of $\frak{V}$ tells us that for any
bundle $p:\Sigma\rightarrow B$ we can also construct a bundle $q:\frak{V}%
_{B}\Sigma\rightarrow B$, by applying $\frak{V}$ to an internal locale in the
topos of sheaves over $B$ got using the Localic Bundle Theorem. But $\frak{V}$
is also geometric, and this tells us that the bundle construction works
fibrewise: in other words, each fibre $b^{\ast}\frak{V}_{B}^{(1)}\Sigma
=q^{-1}(\{b\})$ is homeomorphic to $\frak{V}^{(1)}(b^{\ast}\Sigma)$. We shall
not need to dwell on the topos theory here, but it is the topos theory that
tells us how the valuation locales of the fibres of $p$ can be bundled
together, with an appropriate topology on the bundle space, to make $q$.

\begin{remark}
\label{DecDiscValnRem}In general, to define a valuation on $X$ involves
defining its values for all \emph{opens} of $X$, or at least for a generating
lattice of opens. However, in the particular case where $X$ is discrete and
moreover has decidable equality, it is enough to define the values for all
\emph{points}. (The issue is similar to that well known for Lebesgue measure,
where the points all have measure 0 but opens have non-zero measure.) The
frame of opens, the powerset $\mathcal{P}X$, is the ideal completion of the
Kuratowski finite powerset $\mathcal{F}X$. Each $S\in\mathcal{F}X$ is a finite
disjoint union of singletons (we need decidability of equality to remove
duplicates and hence achieve disjointness), and so its measure is determined
by that of the singletons.
\end{remark}

%%%%%%%%%%%%%%%%%%%%%%%%%%%%%%%%%%%%%%%%%%%%%%%%%%%%%%%%%%%%
%
%
%
%

\section{Spectral and related bundles}

\label{BundlesSpectralAndRelated}

We briefly summarize some topos approaches and how they lead to bundles. For
convenience we shall refer to the \emph{Imperial} approach of Isham and
Butterfield~\cite{IshamBfield:ToposPKST1},~\cite{IshamBfield:ToposPKST2} and
subsequently D\"{o}ring and Isham \cite{DoeringIsham:WhatThingTTFP}; and the
\emph{Nijmegen} approach of Heunen, Landsman and
Spitters~\cite{HeunenLandsmanSpitters:ToposAQT} (see
also~\cite{HeunenLandsmanSpitters:Tovariance}).

Both Imperial and Nijmegen start with a C*-algebra $\mathcal{A}$ (or, more
specifically in the Imperial case, a von Neumann algebra; in the finite
dimensional case these notions are equivalent) and then take a ``classical
point of view'' (to use the Imperial phrase) to be a commutative C*-subalgebra
$C$. By Gelfand-Naimark duality, $C$ is isomorphic to the algebra of
continuous maps $\Sigma_{C}\rightarrow\mathbb{C}$ where $\Sigma_{C}$ is the
spectrum, and it follows that $\Sigma_{C}$ provides a classically pure state
space for the self-adjoint elements of $C$, considered as observables. They
are all represented as maps $\Sigma_{C}\rightarrow\mathbb{R}$. Thus $C$ is a
context in which the physics of those observables is classical.

Let us write $\mathcal{C}(\mathcal{A})$ for the poset of commutative
C*-subalgebras (or, for Imperial, of commutative von Neumann subalgebras). The
base space of the bundle is constructed out of $\mathcal{C}(\mathcal{A})$, and
the fibre over a context $C$ is its spectrum. A significant difference between
the two approaches lies in how those fibres are topologized.

For \emph{Imperial}, the topos is the category of contravariant functors from
$\mathcal{C}(\mathcal{A})$ to $\mathbf{Set}$ (i.e. presheaves over
$\mathcal{C}(\mathcal{A})$). For a correct point-free approach one should take
the base space $B$ to be $\operatorname*{Idl}(\mathcal{C}(\mathcal{A})^{op})$,
the space of filters of $\mathcal{C}(\mathcal{A})$, with its Scott topology.
However, for many purposes it suffices to consider only the \emph{principal}
filters, of the form $\{D\in\mathcal{C}(\mathcal{A})\mid C\subseteq D\}$ for
some $C$. The Imperial workers seek the spectrum as an \emph{object} of the
topos, i.e. a sheaf, corresponding to a bundle that is a local homeomorphism
over $\operatorname*{Idl}(\mathcal{C}(A)^{op})$, and so the fibres are all discrete.

By Theorem~\ref{FibOpfibThm} the spectral bundle is therefore an opfibration,
giving fibre maps covariant with respect to the specialization order. Since
the fibre maps are of necessity contravariant with respect to context
inclusion (by a ``coarse-graining'' argument), it follows that the
specialization order is the opposite of context inclusion. This is achieved in
\cite{DoeringIsham:WhatThingTTFP} by taking the base space to be the filter
completion of the context poset, and its sheaves are then presheaves --
contravariant set-valued functors -- on the context poset. Any approach that
seeks a spectral object, a discrete space internally, will be subject to this
argument, so it is a consequence of using point-set topology internally.

For \emph{Nijmegen}, the topos is the category of covariant functors from
$\mathcal{C}(\mathcal{A})$ to $\mathbf{Set}$ and the corresponding base space
$B$ is $\operatorname*{Idl}\mathcal{C}(\mathcal{A})$, the space of ideals of
$\mathcal{C}(\mathcal{A})$, with its Scott topology. The Nijmegen workers seek
the spectrum as an internal compact regular -- actually, completely regular --
space in the topos, as expected from the Gelfand-Naimark duality. In fact,
they can construct an internal commutative C*-algebra (for which the fibre
over $C$ is just $C$ itself) and then use a topos-valid version
\cite{BanMulv:GlobalisationGDT} of Gelfand-Naimark duality to get an
internally compact regular spectrum. In the corresponding bundle the fibres
are all compact regular. Now by Theorem~\ref{FibOpfibThm} the spectral bundle
is a fibration, and the same reasoning shows that the specialization order
will agree with context inclusion. This is achieved in
\cite{HeunenLandsmanSpitters:Tovariance} by taking the base to be the ideal
completion of the context poset, and the sheaves are the covariant set-valued
functors on the context poset.

Note how the Nijmegen approach had to adopt a point-free approach to topology
in order to use the topos-valid form of Gelfand-Naimark. Whereas a point-set
approach always gives an opfibration, the change to point-free does not in
itself bring any consequences for the variance because a general bundle need
be neither fibration nor opfibration. However, any approach that seeks a
compact, regular, point-free spectrum will get a fibration with the same
variance as for Nijmegen.

Although these two approaches are the best worked out so far, they are not the
only possible. As we shall see in Section~\ref{finDimQM}, in finite
dimensional systems parts of the base and bundle spaces have natural manifold
structure and can therefore be given non-discrete Hausdorff topologies. Our
bundle description is intended to allow a wide range of such possibilities,
when described point-free.

%%%%%%%%%%%%%%%%%%%%%%%%%%%%%%%%%%%%%%%%%%%%%%%%%%%%%%%%%%%%
%
%
%
%

\subsection{Valuation bundles and Born sections}

\label{BundlesValuationBundles}

Suppose $p:\Sigma\rightarrow B$ is our spectral bundle, giving rise to a
probability valuation bundle $q:\frak{V}_{B}^{(1)}\Sigma\rightarrow B$.

For the Nijmegen situation this was discussed in
\cite{HeunenLandsmanSpitters:ToposAQT} (referring to the development
in\cite{CoqSpit:IntVal}), and the paper proves (their Theorem~14) that its
global points -- the global sections of $q$ -- are equivalent to quasistates
of the C$^{\ast}$-algebra. Thus in particular each pure state $|\psi\rangle$
gives a global section of $q$, and it is continuous. This is already
interesting, since Kochen-Specker tells us that global sections of $p$ should
normally be impossible. If we try to extract external mathematics from the
topos internal by taking global sections, then the spectrum $p$ loses all its
points, but the valuation space $q$ retains points and we are familiar with
them through quantum pure states.

Once a section $\sigma$ is given for $q$, we can infer the probabilities that
arise in the Born rule. (See also \cite{Doering:QuStMSP},
\cite{DoeringIsham:ClassicalQPTV} for some other discussions of the Born
rule.) Suppose $C$ is a context, a point of $B$, and for simplicity consider
the simple case where $C$ is a commutative subalgebra (rather than a filter or
ideal). The fibre $C^{\ast}\Sigma$ is the Gelfand spectrum
$\operatorname*{Spec}(C)$, and any self adjoint in $C$ is represented as a map
$\tilde{O}:C^{\ast}\Sigma\rightarrow\mathbb{R}$. Using geometricity to see
$C^{\ast}\frak{V}_{B}^{(1)}\Sigma\cong\frak{V}^{(1)}C^{\ast}\Sigma$, we can
then get the probabilistic result of observing $O$ in state $|\psi\rangle$ as
the probabilistic valuation $\frak{V}^{(1)}(\tilde{O})(\sigma(C))$ on
$\mathbb{R}$, a random variable. The core probability is determined by $C$ and
$|\psi\rangle$, and all that remains is to allow for the way the observable
$O$ labels the eigenstates (in $C^{\ast}\Sigma$) with real numbers.

Our next step is to generalize the state $|\psi\rangle$. In the finite
dimensional case it corresponds to a projector $|\psi\rangle\langle\psi|$ of
rank 1 and it is possible to generalize to higher rank projectors by summing
over basis states. Thus for each context $D$, the elements of $D^{\ast}\Sigma$
can play a role similar to that of states $|\psi\rangle$. However, we do not
attempt to normalize -- to do so presents problems when allowing for
refinement of the context $D$ -- and so we no longer have probability
valuations in general. Thus we get sections of $q:\frak{V}_{B}\Sigma
\rightarrow B$. Bundling these together we postulate a \emph{Born section} of
the bundle $\frak{V}_{B^{2}}\Sigma^{2}\rightarrow B^{2}$. For each pair
$(D,C)$ of contexts it gives a valuation $\mathsf{Born}_{DC}$ of $D^{\ast
}\Sigma\times C^{\ast}\Sigma$.

At present we have the Born sections defined only in finite dimensional
situations. Nonetheless, the formal notion makes sense more generally and
would appear to have a good phenomenological footing in that it describes
probabilities. We therefore hope that the structure of bundle together with
Born maps, appropriately axiomatized, will prove a good foundation for
topos-based contextual physics.

%%%%%%%%%%%%%%%%%%%%%%%%%%%%%%%%%%%%%%%%%%%%%%%%%%%%%%%%%%%%
%
%
%
%

\section{Finite dimensional quantum systems}

\label{finDimQM}

%%%%%%%%%%%%%%%%%%%%%%%%%%%%%%%%%%%%%%%%%%%%%%%%%%%%%%%%%%%%
%
%
%
%

\subsection{Bundles for finite dimensional quantum systems}

\label{findimBundlesSectn}

In the following we propose a fibrational bundle in which the spaces $B$ and
$\Sigma$ contain manifolds. This contrasts with the original Imperial and Nijmegen
constructions, where the topology on $B$ arises solely from the order
structure on $\mathcal{C}(\mathcal{A})$, though it accords with the use of
flag manifolds in \cite{CaspersHLS:IntQLnLS}.
We fix an algebra $\mathcal{A}%
=M_{n}(\mathbb{C})$ and all constructions are relative to this algebra. Any
commutative sub-C$^{\ast}$-algebra $C$ is also finite
dimensional\footnote{Constructively this is not quite true. But we shall
construct our bundle in terms of the subalgebras that \emph{are} finite
dimensional.} and so, by Gelfand-Naimark duality, isomorphic as unital
C$^{\ast}$-algebra to $\mathbb{C}^{m}$ for some $m$. It has $m$ indecomposable
projectors (self-adjoint idempotents) $C_{i}$ corresponding up to isomorphism
to the elements of $\mathbb{C}^{m}$ that have a $1$ in a single component, $0$
elsewhere, and these are also projectors in $\mathcal{A}$ because $C$ is a
sub-C$^{\ast}$-algebra. They are orthogonal ($C_{i}C_{j}=0$ if $i\not =j$) and
complete (they sum to $1$). Note also that the trace of each is a positive
integer, equal to the rank, since the eigenvalues are 0 or 1.

We call a complete orthogonal sequence $\overrightarrow{C}$ of projectors a
\emph{projector system,} and define its \emph{type} to be the sequence of
traces of the projectors, an ordered partition of $n$. Generally we are only
interested in the \emph{set} of projectors (because this is what characterizes
the subalgebra $C$) and the set of traces as type; but for setting up the
bundle it is useful to remember the automorphisms. With this in mind, we
define --

\begin{definition}
Let $(\mu_{i})_{i=1}^{l}$ and $(\nu_{j})_{j=1}^{m}$ be two partitions of $n$.
A \emph{refinement} from $\overrightarrow{\nu}$ to $\overrightarrow{\mu}$ is a
function $r:\{1,\ldots,l\}\rightarrow\{1,\ldots m\}$ such that $\nu_{j}%
=\sum_{r(i)=j}\mu_{i}$ for every $j$. (Note that the reindexing function is in
the opposite direction to the refinement.)
\end{definition}

For each partition $\overrightarrow{\mu}$ we define the space
$\operatorname*{Proj}(\overrightarrow{\mu})$ of projector systems of type
$\overrightarrow{\mu}$. This can clearly be done localically, since it is
defined as a subspace of $\mathbb{C}^{ln^{2}}$, where $l$ is the length of
$\overrightarrow{\mu}$, by a system of equations. We also have a trivial
bundle over it, $\operatorname*{Proj}(\overrightarrow{\mu})\times
l\rightarrow\operatorname*{Proj}(\overrightarrow{\mu})$. Each fibre has
exactly $l$ elements, which is the Gelfand spectrum for all subalgebras of
type $\overrightarrow{\mu}$. For any refinement $r:\overrightarrow{\nu
}\rightarrow\overrightarrow{\mu}$ we now have a map $\operatorname*{Proj}%
(r):\operatorname*{Proj}(\overrightarrow{\mu})\rightarrow\operatorname*{Proj}%
(\overrightarrow{\nu})$ given by $\operatorname*{Proj}(r)(\overrightarrow
{C})_{j}=\sum_{r(i)=j}C_{i}$. This extends to a bundle map using
$\operatorname*{Proj}(r)(\overrightarrow{C},i)=(\operatorname*{Proj}%
(r)(\overrightarrow{C}),r(i))$. We thus have a diagram of bundles, whose shape
is the opposite of the category of partitions and refinements. Our bundle
$\Sigma\rightarrow B$ is now defined as a lax colimit of this diagram. More
precisely, it has images of all the spaces $\operatorname*{Proj}%
(\overrightarrow{\mu})$, subject to $\operatorname*{Proj}(r)(\overrightarrow
{C})\sqsubseteq\overrightarrow{C}$ and $(\operatorname*{Proj}%
(r)(\overrightarrow{C}),r(i))\sqsubseteq(\overrightarrow{C},i)$.

The imposition of this specialization order has two effects. The first, and
perhaps less obvious one, is with regard to invertible refinements. These
permute equal values in partitions, and the effect of the imposed
specialization is to make two projector systems equal if they generate the
same subalgebra -- because they have the same projectors, but permuted. This
makes $B$ a space of contexts as required. After that, the specialization
agrees with context inclusion as required for a fibrational bundle. The action
on the bundle spaces ensures that states are kept track of correctly under
permuting of the matrices.

To define the Born section, we define it first for projector systems, and show
that the definition respects refinements. Suppose $\overrightarrow{C}$ and
$\overrightarrow{D}$ have types $\overrightarrow{\mu}$ and $\overrightarrow
{\nu}$, of lengths $l$ and $m$. At this level (before imposing the
specialization) the spectra are finite discrete with decidable equality,
cardinalities $l$ and $m$, and so by Remark~\ref{DecDiscValnRem} a valuation
on the product can be defined by the values on its elements $(i,j)$. We define%
\[
\beta_{\overrightarrow{C}\overrightarrow{D}}(i,j)=\operatorname*{Tr}%
(C_{i}D_{j})\text{,}%
\]
a non-negative real, and then let $\mathsf{Born}_{CD}$ be the image of
$\beta_{\overrightarrow{C}\overrightarrow{D}}$ in $\frak{V}_{B^{2}}(\Sigma
^{2})$. Now suppose $\overrightarrow{C}\sqsubseteq\overrightarrow{C}^{\prime}$
and $\overrightarrow{D}\sqsubseteq\overrightarrow{D}^{\prime}$, with
refinements $r:\overrightarrow{\mu}\rightarrow\overrightarrow{\mu}^{\prime}$
and $s:\overrightarrow{\nu}\rightarrow\overrightarrow{\nu}^{\prime}$. Then%
\[
\operatorname*{Tr}(C_{i}D_{j})=\operatorname*{Tr}(\sum_{r(i^{\prime}%
)=i}C_{i^{\prime}}^{\prime}\sum_{s(j^{\prime})=j}D_{j^{\prime}}^{\prime}%
)=\sum_{r(i^{\prime})=i,s(j^{\prime})=j}\operatorname*{Tr}(C_{i^{\prime}%
}^{\prime}D_{j^{\prime}}^{\prime})
\]
and so $\beta_{\overrightarrow{C}\overrightarrow{D}}=\frak{V}\left(
\operatorname*{Proj}(r)\times\operatorname*{Proj}(s)\right)  (\beta
_{\overrightarrow{C}^{\prime}\overrightarrow{D}^{\prime}})$; it follows that
$\mathsf{Born}_{CD}\sqsubseteq\mathsf{Born}_{C^{\prime}D^{\prime}}$.

In the case where $C_{i}$ has trace 1 it is of the form $|\psi\rangle
\langle\psi|$ for some unit vector $|\psi\rangle$ (an eigenvector for
eigenvalue 1). Then%
\[
\operatorname*{Tr}(C_{i}D_{j})=\operatorname*{Tr}(|\psi\rangle\langle
\psi|D_{j})=\operatorname*{Tr}(\langle\psi|D_{j}|\psi\rangle)=\langle
\psi|D_{j}|\psi\rangle
\]
and so the probability agrees with that obtained from the Born rule for state
$|\psi\rangle$. More generally, $C_{i}$ is a sum $\sum_{k}|\psi_{k}%
\rangle\langle\psi_{k}|$ for orthonormal vectors, and $\mathsf{Born}%
_{CD}(i,j)$ is the sum $\sum_{k}\langle\psi_{k}|D_{j}|\psi_{k}\rangle$.

%%%%%%%%%%%%%%%%%%%%%%%%%%%%%%%%%%%%%%%%%%%%%%%%%
%
%
%
%
%
%
%
%
%
%

\subsection{Bundles for the qubit}

The qubit Hilbert space $\mathcal{H}$ is $\mathbb{C}^{2}$, and the C$^{\ast}%
$-algebra is the full matrix algebra $\mathcal{A}=M_{2}(\mathbb{C})$. The
commutative context $\ast$-subalgebras come in two types, $(2)$ (dimension 1)
and $(1,1)$ (dimension 2). The sole type (2) algebra is the centre
$\mathbb{C}1$, which we denote by $\bot$; $\operatorname*{Proj}(2)$ is the
1-point space. The 2-dimensional subalgebras are generated by two projectors
$C_{1},C_{2}$ such that $C_{1}+C_{2}=1$ and $\operatorname*{Tr}(C_{i})=1$.
Projectors $P$ of trace 1 are in bijection with self-adjoint unitaries
$U=2P-1$ of trace $0$, and such unitary can be written as a real linear
combination $a_{x}\sigma_{x}+a_{y}\sigma_{y}+a_{z}\sigma_{z}$ of the Paulis
such that $a_{x}^{2}+a_{y}^{2}+a_{z}^{2}=1$. (Any self-adjoint is a real
linear combination of the Paulis and $1$; for trace $0$ the coefficient of $1$
is $0$; and the further condition says that the matrix is an involution.)
Since each $C_{i}$ is determined by the other, it follows that
$\operatorname*{Proj}(1,1)$ is the 2-sphere $S^{2}$ -- this is the Bloch
sphere. In the context space $B$ antipodes will be identified, giving the real
projective plane $\mathbb{R}\mathsf{P}^{2}$, and $\bot$ is adjoined as a
bottom point in the specialization order. Thus $B$ is $\mathbb{R}%
\mathsf{P}^{2}$ lifted, $\Sigma$ is $S^{2}$ lifted.

For both Imperial and Nijmegen (but see also \cite{Spitters:SpaceMOSNCA}),
both $S^{2}$ and $\mathbb{R}P^{2}$ are given
their discrete topologies. For Imperial, the trivial points are adjoined
\emph{above} the rest in the specialization order as a top $\top$, so that
each trivial point is open. For Nijmegen they are adjoined \emph{below} as a
bottom $\bot$, so that the only open containing a trivial point is the whole space.

An interesting consequence of using the natural, manifold topologies, is that
the bundle has no continuous cross-sections. As~\cite{IshamBfield:ToposPKST1}
have pointed out, for dimensions 3 or more the lack of cross-sections for the
Imperial bundle is a manifestation of the Kochen-Specker
Theorem~\cite{KochenSpecker}. That theorem does not apply directly to
dimension 2, and indeed the corresponding Imperial bundle has many
cross-sections, albeit discontinuous with respect to the manifold topologies.
Essentially the Kochen-Specker Theorem as normally formulated is a
combinatorial one, relying on having sufficient complexity in the order
structure amongst the contexts. In the 2-dimensional case that order structure
(one trivial point related to everything, all other points incomparable with
each other) is too simple.

%\bibliographystyle{eptcs} %{amsalpha}
%\bibliography{BornBiblio,MyBiblio}

\begin{thebibliography}{10}
\providecommand{\bibitemdeclare}[2]{}
\providecommand{\urlprefix}{Available at }
\providecommand{\url}[1]{\texttt{#1}}
\providecommand{\href}[2]{\texttt{#2}}
\providecommand{\urlalt}[2]{\href{#1}{#2}}
\providecommand{\doi}[1]{doi:\urlalt{http://dx.doi.org/#1}{#1}}
\providecommand{\bibinfo}[2]{#2}

\bibitemdeclare{unpublished}{AbrBrand:UnifiedSTANLC}
\bibitem{AbrBrand:UnifiedSTANLC}
\bibinfo{author}{Samson Abramsky} \& \bibinfo{author}{Adam Brandenburger}
  (\bibinfo{year}{2011}): \emph{\bibinfo{title}{A Unified Sheaf-Theoretic
  Account of Non-Locality and Contextuality}}.
\newblock \bibinfo{note}{ArXiv:quant-ph/1102.0264v2}.

\bibitemdeclare{article}{BanMulv:GlobalisationGDT}
\bibitem{BanMulv:GlobalisationGDT}
\bibinfo{author}{Bernhard Banaschewski} \& \bibinfo{author}{Christopher~J.
  Mulvey} (\bibinfo{year}{2006}): \emph{\bibinfo{title}{A Globalisation of the
  {G}elfand Duality Theorem}}.
\newblock {\sl \bibinfo{journal}{Annals of Pure and Applied Logic}}
  \bibinfo{volume}{137}, pp. \bibinfo{pages}{62--103},
  \doi{10.1016/j.apal.2005.05.018}.

\bibitemdeclare{article}{CaspersHLS:IntQLnLS}
\bibitem{CaspersHLS:IntQLnLS}
\bibinfo{author}{Martijn Caspers}, \bibinfo{author}{Chris Heunen},
  \bibinfo{author}{Nicolaas~P. Landsman} \& \bibinfo{author}{Bas Spitters}
  (\bibinfo{year}{2009}): \emph{\bibinfo{title}{Intuitionistic Quantum Logic of
  an $n$-level System}}.
\newblock {\sl \bibinfo{journal}{Foundations of Physics}} \bibinfo{volume}{39},
  pp. \bibinfo{pages}{731--759}, \doi{10.1007/s10701-009-9308-7}.

\bibitemdeclare{article}{CoqSpit:IntVal}
\bibitem{CoqSpit:IntVal}
\bibinfo{author}{Thierry Coquand} \& \bibinfo{author}{Bas Spitters}
  (\bibinfo{year}{2009}): \emph{\bibinfo{title}{{Integrals and Valuations}}}.
\newblock {\sl \bibinfo{journal}{Journal of Logic and Analysis}}
  \bibinfo{volume}{1}(\bibinfo{number}{3}), pp. \bibinfo{pages}{1--22},
  \doi{10.4115/jla.2009.1.3}.

\bibitemdeclare{article}{Doering:QuStMSP}
\bibitem{Doering:QuStMSP}
\bibinfo{author}{Andreas D\"{o}ring} (\bibinfo{year}{2009}):
  \emph{\bibinfo{title}{Quantum States and Measures on the Spectral Presheaf}}.
\newblock {\sl \bibinfo{journal}{Advanced Science Letters}}
  \bibinfo{volume}{2}(\bibinfo{number}{2}), pp. \bibinfo{pages}{291--301}.

\bibitemdeclare{unpublished}{DoeringIsham:ClassicalQPTV}
\bibitem{DoeringIsham:ClassicalQPTV}
\bibinfo{author}{Andreas D\"{o}ring} \& \bibinfo{author}{Chris Isham}
  (\bibinfo{year}{2011}): \emph{\bibinfo{title}{Classical and Quantum
  Probabilities as Truth Values}}.
\newblock \bibinfo{note}{ArXiv:1102.2213}.

\bibitemdeclare{incollection}{DoeringIsham:WhatThingTTFP}
\bibitem{DoeringIsham:WhatThingTTFP}
\bibinfo{author}{Andreas D\"{o}ring} \& \bibinfo{author}{Chris Isham}
  (\bibinfo{year}{2011}): \emph{\bibinfo{title}{`{W}hat is a {T}hing?': Topos
  Theory in the Foundations of Physics}}.
\newblock In \bibinfo{editor}{Bob Coecke}, editor: {\sl
  \bibinfo{booktitle}{{New Structures for Physics}}},
  chapter~\bibinfo{chapter}{13}, {\sl \bibinfo{series}{Lecture Notes in
  Physics}} \bibinfo{volume}{813}, \bibinfo{publisher}{Springer Verlag},
  \bibinfo{address}{Berlin Heidelberg}, pp. \bibinfo{pages}{753--937},
  \doi{10.1007/978-3-642-12821-9\_13}.
\newblock \bibinfo{note}{ArXiv:0803.0417v1.}

\bibitemdeclare{inproceedings}{FourmanScott:SheavesLogic}
\bibitem{FourmanScott:SheavesLogic}
\bibinfo{author}{M.P. Fourman} \& \bibinfo{author}{D.S. Scott}
  (\bibinfo{year}{1979}): \emph{\bibinfo{title}{Sheaves and Logic}}.
\newblock In \bibinfo{editor}{M.P. Fourman}, \bibinfo{editor}{C.J. Mulvey} \&
  \bibinfo{editor}{D.S. Scott}, editors: {\sl \bibinfo{booktitle}{Applications
  of Sheaves}}, {\sl \bibinfo{series}{Lecture Notes in Mathematics}}
  \bibinfo{volume}{753}, \bibinfo{publisher}{Springer-Verlag}, pp.
  \bibinfo{pages}{302--401}, \doi{10.1007/BFb0061824}.

\bibitemdeclare{incollection}{Giry:CatProb}
\bibitem{Giry:CatProb}
\bibinfo{author}{M.~Giry} (\bibinfo{year}{1981}): \emph{\bibinfo{title}{A
  Categorical Approach to Probability Theory}}.
\newblock In: {\sl \bibinfo{booktitle}{Categorical Aspects of Topology and
  Analysis}}, {\sl \bibinfo{series}{Lecture Notes in Mathematics}}
  \bibinfo{volume}{915}, \bibinfo{publisher}{Springer-Verlag},
  \bibinfo{address}{Berlin}, pp. \bibinfo{pages}{68--85},
  \doi{10.1007/s11225-010-9232-z}.

\bibitemdeclare{inproceedings}{Heckmann:ProbPDISL}
\bibitem{Heckmann:ProbPDISL}
\bibinfo{author}{Reinhold Heckmann} (\bibinfo{year}{1994}):
  \emph{\bibinfo{title}{Probabilistic Power Domains, Information Systems, and
  Locales}}.
\newblock In \bibinfo{editor}{Stephen Brookes}, \bibinfo{editor}{Michael Main},
  \bibinfo{editor}{Austin Melton} \& \bibinfo{editor}{David Schmidt}, editors:
  {\sl \bibinfo{booktitle}{Mathematical Foundations in Programming Semantics --
  9th International Conference, 1993}}, {\sl \bibinfo{series}{Lecture Notes in
  Computer Science}} \bibinfo{volume}{802}, \bibinfo{publisher}{Springer},
  \bibinfo{address}{Berlin / Heidelberg}, pp. \bibinfo{pages}{410--437}.

\bibitemdeclare{inproceedings}{HeunenLandsmanSpitters:Tovariance}
\bibitem{HeunenLandsmanSpitters:Tovariance}
\bibinfo{author}{Chris Heunen}, \bibinfo{author}{Nicolaas~P. Landsman} \&
  \bibinfo{author}{Bas Spitters} (\bibinfo{year}{2008}):
  \emph{\bibinfo{title}{The Principle of General Tovariance}}.
\newblock In \bibinfo{editor}{Rui~Loja Fernandez} \& \bibinfo{editor}{Roger
  Picken}, editors: {\sl \bibinfo{booktitle}{Geometry and Physics, {X}{V}{I}
  {I}nternational {F}all {W}orkshop, {L}isbon, Portugal, 5-8 September 2007}},
  {\sl \bibinfo{series}{AIP Conference Proceedings, Mathematical and
  Statistical Physics}} \bibinfo{volume}{1023}, \bibinfo{publisher}{Springer}.

\bibitemdeclare{article}{HeunenLandsmanSpitters:ToposAQT}
\bibitem{HeunenLandsmanSpitters:ToposAQT}
\bibinfo{author}{Chris Heunen}, \bibinfo{author}{Nicolaas~P. Landsman} \&
  \bibinfo{author}{Bas Spitters} (\bibinfo{year}{2009}):
  \emph{\bibinfo{title}{A Topos for Algebraic Quantum Theory}}.
\newblock {\sl \bibinfo{journal}{Communications in Mathematical Physics}}
  \bibinfo{volume}{291}(\bibinfo{number}{1}), pp. \bibinfo{pages}{63--110},
  \doi{10.1007/s00220-009-0865-6}.

\bibitemdeclare{unpublished}{HeunenLandsmanSpitters:Bohrn}
\bibitem{HeunenLandsmanSpitters:Bohrn}
\bibinfo{author}{Chris Heunen}, \bibinfo{author}{Nicolaas~P. Landsman} \&
  \bibinfo{author}{Bas Spitters} (\bibinfo{year}{2010}):
  \emph{\bibinfo{title}{{B}ohrification of Operator Algebras and Quantum
  Logic}}.
\newblock \bibinfo{note}{ArXiv:quant-ph/0905.2275v4}.

\bibitemdeclare{article}{IshamBfield:ToposPKST1}
\bibitem{IshamBfield:ToposPKST1}
\bibinfo{author}{C.J. Isham} \& \bibinfo{author}{J.~Butterfield}
  (\bibinfo{year}{1998}): \emph{\bibinfo{title}{Topos Perspectives on the
  {K}ochen-{S}pecker {T}heorem {I}. Quantum States as Generalized Valuations}}.
\newblock {\sl \bibinfo{journal}{International Journal of Theoretical Physics}}
  \bibinfo{volume}{37}(\bibinfo{number}{11}), pp. \bibinfo{pages}{2669--2733},
  \doi{10.1023/A:1026680806775}.

\bibitemdeclare{article}{IshamBfield:ToposPKST2}
\bibitem{IshamBfield:ToposPKST2}
\bibinfo{author}{C.J. Isham} \& \bibinfo{author}{J.~Butterfield}
  (\bibinfo{year}{1999}): \emph{\bibinfo{title}{Topos Perspectives on the
  {K}ochen-{S}pecker {T}heorem {I}{I}. Conceptual Aspects and Classical
  Analogues}}.
\newblock {\sl \bibinfo{journal}{International Journal of Theoretical Physics}}
  \bibinfo{volume}{38}(\bibinfo{number}{3}), pp. \bibinfo{pages}{827--859},
  \doi{10.1023/A:1026652817988}.

\bibitemdeclare{article}{Jo:FibPP}
\bibitem{Jo:FibPP}
\bibinfo{author}{P.T. Johnstone} (\bibinfo{year}{1993}):
  \emph{\bibinfo{title}{Fibrations and Partial Products in a 2-category}}.
\newblock {\sl \bibinfo{journal}{Applied Categorical Structures}}
  \bibinfo{volume}{1}, pp. \bibinfo{pages}{141--179}, \doi{10.1007/BF00880041}.

\bibitemdeclare{inproceedings}{JonesPlot:ProbPower}
\bibitem{JonesPlot:ProbPower}
\bibinfo{author}{C.~Jones} \& \bibinfo{author}{G.~Plotkin}
  (\bibinfo{year}{1989}): \emph{\bibinfo{title}{A Probabilistic Powerdomain of
  Evaluations}}.
\newblock In: {\sl \bibinfo{booktitle}{LICS '89}}, \bibinfo{publisher}{IEEE
  Computer Society Press}, pp. \bibinfo{pages}{186--195}.

\bibitemdeclare{article}{JoyalTier}
\bibitem{JoyalTier}
\bibinfo{author}{A.~Joyal} \& \bibinfo{author}{M.~Tierney}
  (\bibinfo{year}{1984}): \emph{\bibinfo{title}{An Extension of the {G}alois
  Theory of {G}rothendieck}}.
\newblock {\sl \bibinfo{journal}{Memoirs of the American Mathematical Society}}
  \bibinfo{volume}{309}.

\bibitemdeclare{article}{KochenSpecker}
\bibitem{KochenSpecker}
\bibinfo{author}{Simon Kochen} \& \bibinfo{author}{Ernst Specker}
  (\bibinfo{year}{1967}): \emph{\bibinfo{title}{{The problem of hidden
  variables in quantum mechanics}}}.
\newblock {\sl \bibinfo{journal}{Journal of Mathematics and Mechanics}}
  \bibinfo{volume}{17}, pp. \bibinfo{pages}{59--87}.

\bibitemdeclare{inproceedings}{Spitters:SpaceMOSNCA}
\bibitem{Spitters:SpaceMOSNCA}
\bibinfo{author}{Bas Spitters} (\bibinfo{year}{2010}):
  \emph{\bibinfo{title}{The space of Measurement Outcomes as a Spectrum for
  Non-Commutative Algebras}}.
\newblock In \bibinfo{editor}{S.~Barry Cooper}, \bibinfo{editor}{Prakash
  Panangaden} \& \bibinfo{editor}{Elham Kashefi}, editors: {\sl
  \bibinfo{booktitle}{Proceedings Sixth Workshop on Developments in
  Computational Models: Causality, Computation, and Physics (DCM 2010)}}, {\sl
  \bibinfo{series}{Electronic Proceedings in Theoretical Computer
  Science}}~\bibinfo{volume}{26}, pp. \bibinfo{pages}{127--134},
  \doi{10.4204/EPTCS.26.12}.

\bibitemdeclare{inproceedings}{Street}
\bibitem{Street}
\bibinfo{author}{Ross Street} (\bibinfo{year}{1974}):
  \emph{\bibinfo{title}{Fibrations and {Y}oneda's Lemma in a 2-category}}.
\newblock In \bibinfo{editor}{G.M. Kelly}, editor: {\sl
  \bibinfo{booktitle}{Category Seminar Sydney 1972/73}}, {\sl
  \bibinfo{series}{Lecture Notes in Mathematics}} \bibinfo{volume}{420}, pp.
  \bibinfo{pages}{104--133}.

\bibitemdeclare{article}{Integration}
\bibitem{Integration}
\bibinfo{author}{Steven Vickers} (\bibinfo{year}{2008}):
  \emph{\bibinfo{title}{A Localic Theory of Lower and Upper Integrals}}.
\newblock {\sl \bibinfo{journal}{Mathematical Logic Quarterly}}
  \bibinfo{volume}{54}(\bibinfo{number}{1}), pp. \bibinfo{pages}{109--123},
  \doi{10.1002/malq.200710028}.

\bibitemdeclare{unpublished}{Vickers:ContIsGeom}
\bibitem{Vickers:ContIsGeom}
\bibinfo{author}{Steven Vickers} (\bibinfo{year}{2011}):
  \emph{\bibinfo{title}{Continuity is Geometricity}}.
\newblock \bibinfo{note}{Talk given at workshop ``Logic, Categories,
  Semantics'' at Bordeaux, November 2010. Preprint at
  \url{http://www.cs.bham.ac.uk/~sjv/GeoAspects.pdf}}.

\bibitemdeclare{unpublished}{Vickers:Riesz}
\bibitem{Vickers:Riesz}
\bibinfo{author}{Steven Vickers} (\bibinfo{year}{2011}):
  \emph{\bibinfo{title}{A Monad of Valuation Locales}}.
\newblock \bibinfo{note}{Preprint at
  \url{http://www.cs.bham.ac.uk/~sjv/Riesz.pdf}}.

\end{thebibliography}

\end{document}